\documentclass{amsart}
\usepackage{amsmath,amsfonts,amssymb,amscd,amsthm,epsf}
\usepackage{pinlabel}
\newtheorem{prop}{Proposition}[section]
\newtheorem{thm}[prop]{Theorem}
\newtheorem{cor}[prop]{Corollary}
\newtheorem{ques}[prop]{Question}

\theoremstyle{definition}
\newtheorem{de}[prop]{Definition}

\theoremstyle{remark}
\newtheorem{Remarks}[prop]{Remarks}             

\def\CP{{\mathbb C \mathbb P}}
\def\Z{{\mathbb Z}}

\def\inter{\mathop{\rm int}\nolimits}

\def\id{\mathop{\rm id}\nolimits}

\def\SO{\mathop{\rm SO}\nolimits}

\def\co{\colon\thinspace}

\def\p{\mathop{\rm point}\nolimits}
\begin{document}
\title{Infinite order corks}
\author{Robert E. Gompf}
\address{The University of Texas at Austin, Mathematics Department RLM 8.100, Attn: Robert Gompf,
2515 Speedway Stop C1200, Austin, Texas 78712--1202}
\email{gompf@math.utexas.edu}
\begin{abstract} We construct a compact, contractible 4--manifold $C$, an infinite order self-diffeomorphism $f$ of its boundary, and a smooth embedding of $C$ into a closed, simply connected 4--manifold $X$, such that the manifolds obtained by cutting $C$ out of $X$ and regluing it by powers of $f$ are all pairwise nondiffeomorphic. The manifold $C$ can be chosen from among infinitely many homeomorphism types, all obtained by attaching a 2--handle to the meridian of a thickened knot complement.
\end{abstract}
\maketitle


\section{Introduction}

The wild proliferation of exotic smoothings of 4--manifolds highlights the failure of high--dimensional topology to apply in dimension 4, notably through failure of the h--Cobordism Theorem. Attempts to understand this issue led to the notion of a {\em cork twist}. A cork, as originally envisioned, is a contractible, smooth submanifold $C$ of a closed 4--manifold $X$, with an involution $f$ of $\partial C$, such that cutting out $C$ and regluing it by the twist $f$ changes the diffeomorphism type of $X$ (while necessarily preserving its homeomorphism type). We can think of $C$ as a control knob with two settings, toggling between two smoothings of $X$. The first example of a cork was discovered by Akbulut \cite{A}. Subsequently, various authors (\cite{CFHS}, \cite{M}, see \cite{GS} for more history) showed that any two homeomorphic, simply connected (smooth) 4--manifolds are related by a cork twist. Since then, much work has been done (eg \cite{AR}, \cite{AY}) to understand and apply cork twists. Various people, going back at least to Freedman in the 1990s, have asked whether higher order corks may exist---that is, knobs with $n$ settings for $n$ different diffeomorphism types, or possibly even infinitely many settings all realizing distinct types. Recently, progress has been made by modifying known examples of corks: Tange \cite{T} exhibited knobs with $n$ settings for any finite $n$, displaying two diffeomorphism types on $X$. Independently, Auckly, Kim, Melvin and Ruberman \cite{AKMR} constructed the desired finite order corks. More generally, they constructed $G$--corks for any finite subgroup $G$ of $\SO(4)$, where the control knob can be set to any element of $G$ to yield $|G|$ diffeomorphism types. However, both of these latter papers pose the infinite order case as a still unsolved problem, in spite of fruitless attacks by various mathematicians. The purpose of the present article is to exhibit a large family of infinite order corks, arising from a simple general construction.

There is variation in the literature about the definition of a cork. All approaches share the following:

\begin{de} A {\em cork} $(C,f)$ is a smooth, compact, contractible 4--manifold $C$ with a diffeomorphism $f\co\partial C\to\partial C$. The cork will be called {\em nontrivial} if $f$ does not extend to a self-diffeomorphism of $C$. If $C$ is smoothly embedded in a 4--manifold $X$, cutting out $C$ and regluing it by $f$ to get $(X-\inter C)\cup_f C$ will be called a {\em twist} by $f$. 
\end{de}

\noindent Note that $(C,f^k)$ is then a cork for any $k\in\Z$, so we also talk about twisting by powers $f^k$. By Freedman \cite{F}, \cite{FQ}, $f$ necessarily extends to a self-homeomorphism of $C$, so a cork twist does not change the homeomorphism type of a manifold. In some references, $f$ is required to be an involution, or extend to a finite cyclic (or other finite) group action on $\partial C$. Since we are interested in $\Z$--actions, no additional hypothesis is needed. We can now state our main existence theorem, which is proved in Section~\ref{Cons}.

\begin{thm}\label{main} There is a cork $(C,f)$ and a smooth embedding of $C$ into a closed, simply connected 4--manifold $X$, for which the manifolds $X_k$, $k\in\Z$, obtained by twisting by $f^k$ are homeomorphic but pairwise nondiffeomorphic. Hence, the corks $(C,f^k)$ are distinct (up to diffeomorphism commuting with the maps), and nontrivial unless $k=0$.
\end{thm}

\noindent In the terminology of \cite{AKMR}, the embedding $C\hookrightarrow X$ is {\em $\Z$--effective} and exhibits $(C,f)$ as the first example of a {\em $\Z$--cork}.

\begin{cor} The homology 3--sphere $\partial C$ bounds infinitely many smooth, contractible manifolds that are all diffeomorphic, homeomorphic rel boundary and pairwise nondiffeomorphic rel boundary.
\end{cor}

\begin{proof}
Identify $\partial C$ as the boundary of $C$ using each of the diffeomorphisms $f^k$.
\end{proof} 

\begin{cor} There is a compact, contractible 4--manifold admitting infinitely many nondiffeomorphic smooth structures.
\end{cor}

\begin{proof}
This follows immediately from the previous corollary and Akbulut and Ruberman \cite{AR}.
\end{proof} 

We obtain infinitely many examples of corks $(C,f)$ as in the theorem, distinguished by the homeomorphism types of their boundaries. However, our examples all have a simple form. For any knot $\kappa\subset S^3$, let $P$ be its closed complement, and let $C(\kappa,m)$ be the oriented 4--manifold obtained from $I\times P$, where $I$ denotes the interval $[-1,1]$ throughout the paper, by attaching a 2--handle along the meridian to $\kappa$ in $\{1\}\times P$ with framing $m$. Note that $I\times P$ can be identified with the obvious ribbon complement of $\kappa\#\overline\kappa$ in $B^4$, so this is a special case of removing a slice disk and regluing it with a twist. Either perspective reveals the identity $C(\overline\kappa,m)\approx C(\kappa,m)$, and these are orientation-reversing diffeomorphic to $C(\kappa,-m)$.  Clearly, $C(\kappa,m)$ is the 4--ball when $m=0$ or $\kappa$ is unknotted, but otherwise it is a contractible manifold whose boundary is irreducible and not $S^3$. In fact, $\partial C(\kappa,m)$ is obtained by $(-\frac{1}{m})$--surgery on $\kappa\#\overline \kappa$, and contains two oppositely oriented copies of the complement $P$. When $\kappa$ is prime, the JSJ decomposition of $\partial C(\kappa,m)$ begins by splitting out these complements. (This gives the entire decomposition unless $\kappa$ is a satellite knot, in which case the splitting continues symmetrically.) Since the complements can then be recovered from $\partial C(\kappa,m)$, it follows that the manifolds $C(\kappa,m)$ ($\kappa$ prime) are never orientation-preserving homeomorphic unless the corresponding knots are the same up to orientation and the (nonzero) integers are equal. In our examples, $\kappa$ is the double twist knot $\kappa(r,-s)=\kappa(-s,r)$ shown in Figure~\ref{knot}, where the boxes count full twists, right--handed when the integer is positive. The resulting oriented 4--manifolds $C(r,s;m)=C(\kappa(r,-s),m)$, for $r,s>0$ and $m\ne0$, are not orientation-preserving homeomorphic to each other unless the integers $m$ agree and the pairs $(r,s)$ agree up to order. In general, the incompressible torus $\{0\}\times\partial P$ in $\partial C(\kappa,m)$ can be used to create self-diffeomorphisms of the latter: Let $f\co\partial C(\kappa,m)\to\partial C(\kappa,m)$ be obtained by rotating the torus $\{t\}\times\partial P$ parallel to the canonical longitude of $\kappa$, through angle $( t+1)\pi$, $t\in I=[-1,1]$, as we pass through $I\times\partial P$, and extending as the identity. Our simplest cork, $(C(1,1;-1),f)$ is made in this manner from the figure--eight knot $\kappa(1,-1)$. Its boundary is given by surgery with coefficient 1 on the connected sum of two figure--eight knots, with the obvious incompressible torus in the complement of this sum. More generally, we have:

\begin{figure}
\labellist
\small\hair 2pt
\pinlabel $-s$ at 70 86
\pinlabel $r$ at 70 15
\endlabellist
\centering
\includegraphics{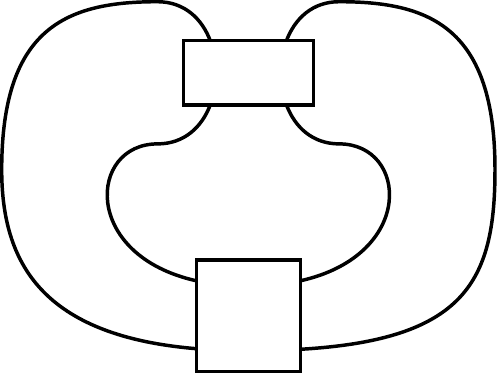}
\caption{The double twist knot $\kappa(r,-s)$.}
\label{knot}
\end{figure}

\begin{thm}\label{corks}
The cork $C$ appearing in Theorem~\ref{main} can be taken to be any of the infinitely many contractible manifolds $C(r,s;m)$ with $r,s>0$ and $m\ne0$, and $f$ as specified above. The manifolds $X_k$ can be assumed to be irreducible, except possibly if $r$, $s$ or $|m|$ equals 2.
\end{thm}

\noindent Recall that a 4--manifold is {\em irreducible} if it cannot split as a smooth connected sum unless one summand is homeomorphic to $S^4$. Other explicit constructions of corks in the literature typically involve reducible (blown up) manifolds. It seems likely that the restriction avoiding 2 is unnecessary; see Remark~\ref{even}(a).

Our incompressible torus in $C(\kappa,m)$ can be also used to define other twists. Instead of twisting parallel to the longitude, we could twist parallel to the meridian, or more generally, twist using any element of $H_1(T^2)\cong\Z\oplus\Z$. Thus, it is natural to ask both about other contractible manifolds and other twists:

\begin{ques} (a) Is every pair $(C(\kappa,m),f)$, for $\kappa$ a nontrivial knot, $m\ne0$ and $f$ a longitudinal twist as given above, a $\Z$--cork? 

(b) Does twisting by other elements of $H_1(T^2)$ ever extend these to $(\Z \oplus\Z)$--corks?
\end{ques}

\noindent Akbulut, in a preliminary version of \cite{A2}, previously studied the meridian twist for $\kappa$ the trefoil and $m=-1$, trying to prove nontriviality. However, we show in \cite{Cork2} that the meridian twist extends over every $C(\kappa,\pm 1)$. Recently, Ray and Ruberman \cite{RR} answered (a) in the negative for torus knots $\kappa$ when $|m|=1$. It follows that every boundary diffeomorphism extends over $C(\kappa,\pm1)$ for such knots \cite{Cork2}. See the latter paper for further discussion and the translation of the main proofs of this paper into the language of handle calculus. The question is still open for meridian twists when $|m|\ge 2$ and for longitudinal twists with nontorus knots $\kappa$ outside our family $\{\kappa(r,-s)|r,s>0\}$.

More recently, Tange has posted papers extending the methods of this article to exhibit $n$--fold boundary sums of our $\Z$--corks as $\Z^n$--corks \cite{T3} and providing constraints on families of manifolds that can be related by $\Z$--corks \cite{T2}.

\section*{Acknowledgements}
The author would like to thank Cameron Gordon, Danny Ruberman and Andr\'as Stipsicz for helpful comments.


\section{Constructing corks}\label{Cons}

The closed manifolds $X_k$ in Theorem~\ref{main} are made from the elliptic surface $E(n)$, for a fixed $n\ge1$, by the Fintushel--Stern knot construction \cite{FS}. Recall (eg \cite{GS}) that $E(n)$ has a standard description in which it is built from $S^1\times S^1\times D^2$ (a neighborhood of a regular fiber $F=S^1\times S^1$) by adding handles. Of most interest for present purposes, each of the two circle factors has $6n$ parallel copies ({\em vanishing cycles}) to which 2--handles are attached with framing $-1$ (relative to the product framing of the boundary 3--torus). We will use three of these 2--handles. Given a knot $K\subset S^3$, let $M$ denote its closed complement. The knot construction consists of removing  $S^1\times S^1\times D^2$ from $E(n)$ and replacing it by $M\times S^1$, gluing by a diffeomorphism of the boundary 3--torus that identifies the canonical longitude of $K$ with $\{\p\}\times\partial D^2$, and the meridian of $K$ and circle $\{\p\}\times S^1$ in $M\times S^1$ with copies of the two circle factors of $F$. As detailed in \cite{FS}, Freedman's classification \cite{F}, \cite{FQ} shows that the resulting manifold $X_K$ retains the homeomorphism type of $E(n)$ (which is simply connected with $b_2=12n-2$ and signature $-8n$, and is even if and only if $n$ is). However when $n\ge2$, varying the knot $K$ results in diffeomorphism types that are distinguished by their Seiberg--Witten invariants if and only if the knots in question are distinguished by their Alexander polynomials. The structure of the Seiberg--Witten invariants then also shows that each $X_K$ is irreducible. When $n=1$ the discussion becomes more technical, but these statements remain true for the $k$--twist knots $K_k=\kappa(k,-1)$ with $k\in\Z$ \cite{FS2}, except that the unknot $K_0$ yields the reducible manifold $E(1)$, a sum of copies of $\pm \CP^2$. For fixed $n\ge1$, let $X_k$ be obtained as above from the twist knot $K_k$.  Since these knots are distinguished by their Alexander polynomials, Theorem~\ref{main} follows once we locate a contractible $C\subset X=X_0=E(n)$ with a twist $f$ for which each power $f^k$ gives the corresponding $X_k$.

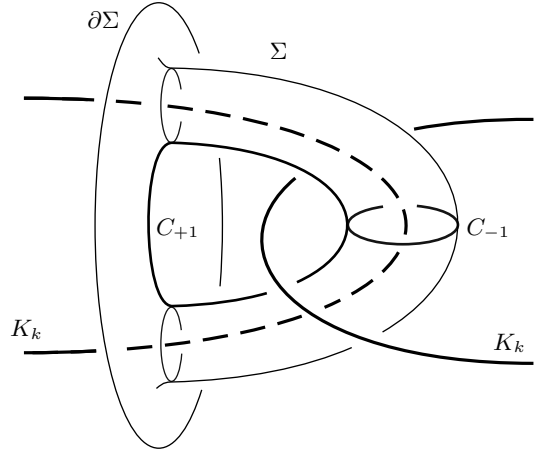
\begin{figure}
\labellist
\small\hair 2pt
\pinlabel $\partial\Sigma$ at 30 162
\pinlabel $\Sigma$ at 96 150
\pinlabel $C_{-1}$ at 175 83
\pinlabel $C_{+1}$ at 58 83
\pinlabel $K_k$ at 183 40
\pinlabel $K_k$ at 1 45
\endlabellist
\centering
\includegraphics{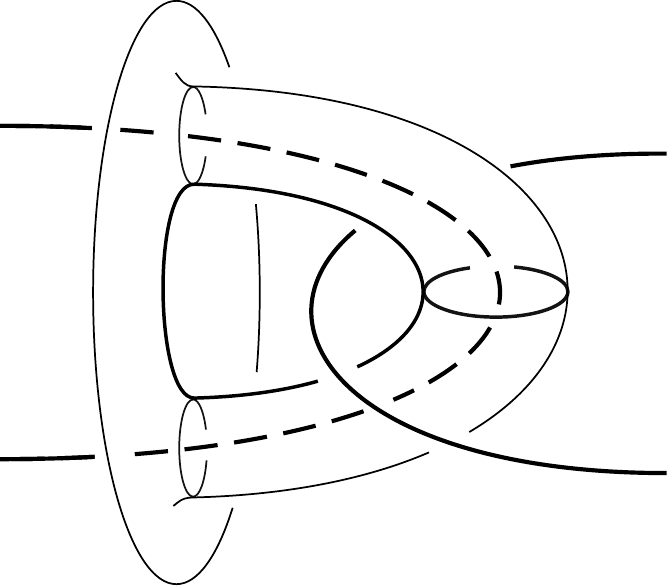}
\caption{The punctured torus $\Sigma$ in the knot complement $M=S^3-\nu K_k$ ($k=0$).}
\label{torus}
\end{figure}

\begin{proof}[Proof of Theorem~\ref{main}]
Let $\Sigma\subset M$ be the punctured torus depicted in Figure~\ref{torus} (near the clasp of $K_k$, the $-s=-1$ twist box in Figure~\ref{knot}) with circles $C_{\pm1}$ generating its homology. Set $k=0$ for this, but note that the corresponding picture for any $k$ (with the correct longitude) is then obtained from the $k=0$ case by $(-\frac{1}{k})$--surgery on the circle $\partial\Sigma$. To examine this surgery more closely, identify a tubular neighborhood of $\Sigma$ in $M$ with $I\times \Sigma$, where $I=[-1,1]$ and $\{1\}\times \Sigma$ contains the outer part of the boundary of $I\times \Sigma$ visible in the figure. Let $A$ be a collar of $\partial \Sigma$ in $\Sigma$. Then we can perform the required surgery by cutting out and regluing the solid torus $I\times A$. Since the surgery coefficient has numerator 1, we can take the gluing diffeomorphism to be the identity everywhere except on the annulus $I\times \partial\Sigma$. That is, we get from the $k=0$ case to the case of arbitrary $k$ by slitting $M$ open along the annulus  $I\times \partial\Sigma$ and regluing by $g^k$ for a suitable Dehn twist $g$ of the annulus. Hence, to transform $X_0$ to $X_k$, we slit $X_0$ open along the 3--manifold  $N=I\times \partial\Sigma\times S^1\subset M\times S^1$ and reglue by $(g\times\id_{S^1})^k$. (This operation can be viewed as a torus surgery, also called a logarithmic transformation, and would be a Luttinger surgery if $\{0\}\times\partial \Sigma\times S^1$ could be made Lagrangian. The latter is ruled out, however, since $X_k$ has no symplectic structure unless $|k|\le1$.) Our goal is to find a contractible manifold $C\subset X_0$ whose boundary contains $N$. Extending $g\times\id_{S^1}$ as the identity over the rest of $\partial C$ then gives the required diffeomorphism $f$ completing the proof.

Our first approximation to $C$ is the manifold $Y=I\times \Sigma\times S^1\subset M\times S^1\subset X_0$. Then $\partial Y$ clearly contains $N$, but $Y$ is far from being contractible. In fact, $Y$ is homotopy equivalent to $(S^1\vee S^1)\times S^1$, so it has $b_1=3$ and $b_2=2$, but no higher dimensional homology. Its fundamental group is generated by three circles $C^*_i$, $i=-1,0,1$ (suitably attached to the base point), where for $i=\pm1$, $C^*_i=\{i\}\times C_i\times \{\theta_i\}$ and $C^*_0=\{1\}\times\{p\}\times S^1$, for distinct points $\theta_{\pm 1}\in S^1$ and $p\in \inter\Sigma-(C_{-1}\cup C_1)$. A basis for $H_2(Y)$ is given by the pair of tori $T_i=\{0\}\times C'_i\times S^1$, $i=\pm1$, where $C'_i$ is parallel to $C_i$ in $\Sigma-\{p\}$. To improve $Y$, observe in Figure~\ref{torus} that the circles $C_{\pm1}$ in $M$ are both meridians of the knot $K_0$. Thus, the knot construction matches all three circles $C^*_i\subset\partial Y$ with vanishing cycles of $E(n)$. We obtain a new manifold $Y'\subset X_0$ by ambiently attaching a $(-1)$--framed 2--handle $h_i$ to $Y$ along $C^*_i$ for each $i=-1,0,1$. Then $Y'$ is simply connected with $T_i$, $i=\pm1$, still giving a basis for $H_2(Y')$, and $N$ still contained in $\partial Y'$. To eliminate the last homology, note that for $i=\pm1$, the core of the handle $h_i$ fits together with the annulus $I\times C_i\times\{\theta_i\}$, forming a pair of disks $D_i$ disjointly embedded rel boundary in $Y'$ (with $\partial D_i=\{-i\}\times C_i\times\{\theta_i\}$). Since each $D_i\cap T_i$ is empty, and $D_i\cap T_{-i}$ is a single point of transverse intersection, deleting tubular neighborhoods of these disks from $Y'$ gives a manifold $C$ with no homology. To see that $\pi_1(C)$ vanishes, use the core of the 2--handle $h_0$ to surger the tori $T_i$ to immersed spheres, without changing the intersections with the disks $D_i$. These spheres then provide nullhomotopies for the meridians of the disks. Thus, $C$ is a contractible manifold whose boundary contains $N$, as required.
\end{proof}

\begin{proof}[Proof of Theorem~\ref{corks}]
To identify the cork $C$ constructed in the proof of Theorem~\ref{main}, first consider any framed sphere $S$ in a manifold $Q$. If we add a handle $h$ to $I\times Q$ along ${1}\times S$, and then delete a neighborhood of the core of $h$, extended down to $\{-1\}\times Q$ using the annulus $I\times S$, the result is easily seen to be $I\times P$, where $P$ is made from $Q$ by surgery on $S$. We apply this trick with $Q=\Sigma\times S^1$ from the previous proof.  Attaching the handles $h_{\pm1}$ to $Y=I\times Q$ and deleting their cores $D_{\pm1}$ gives a manifold of the form $I\times P$ that will become $C$ when $h_0$ is attached. The manifold $P$ is obtained from $Q$ by surgery on the disjoint curves $C_{\pm1}\times\{\theta_{\pm1}\}$, with the framings induced from their identification with vanishing cycles of $E(n)$. These framings are $-1$ relative to the oriented boundary of the fiber neighborhood in $E(n)$ on which we performed the Fintushel--Stern construction, and hence, relative to $\partial Y$. However, the circles $C^*_{\pm1}$ lie on opposite faces of $Y$ (with $I$ coordinate $\pm1$), which inherit opposite orientations from $Q$. Thus, the framing coefficients are $\mp1$ relative to $Q$. To construct a surgery diagram of $Q$, we cap off $\Sigma$ to get an embedding $Q=\Sigma\times S^1\subset T^2\times S^1=T^3$, with the latter exhibited as 0--surgery on the Borromean rings $B$. To recover $Q$, we remove its complementary solid torus in $T^3$. This has the effect of undoing one Dehn filling, leaving one component of $B$ unfilled. The curves $\partial \Sigma\times\{\theta\}$ correspond to canonical longitudes of this drilled out link component, and $\{p\}\times S^1$ is a meridian of it. The surgery curves $C_{\pm1}\times\{\theta_{\pm1}\}$ are then $\mp1$--framed meridians of the other two components. Blowing down changes the unfilled curve into a figure--eight knot $\kappa(1,-1)$ in $S^3$, whose complement is $P$. Attaching $h_0$ to $I\times P$ along $C_0^*$ now gives $C=C(1,1;-1)$, and $\partial \Sigma\times S^1$ is identified with the incompressible torus boundary of the figure--eight complement inside $\partial C$, with $f$ twisting longitudinally as required.

Now that we have realized $C(1,1;-1)$ as the cork $C$ in Theorem~\ref{main}, using 4--manifolds $X_k$ generated from $E(n)$ (so irreducible except for $X_0$ when $n=1$), we can easily realize any $C(r,s;m)$ with $r,s>0>m$ by giving up irreducibility: Just blow up points on the cores of the handles $h_i$ to suitably lower their framings (as measured in $E(n)$). This replaces the original manifolds $X_k$ by their $(r+s+|m|-3)$--fold blowups, which remain pairwise nondiffeomorphic. To realize $m>0$, simply reverse the orientation on each $X_k$. Retaining irreducibility is no harder when the integers $r$, $s$ and $m$ are all odd. Simply choose $n$ large enough that $E(n)$ contains $\frac12(r+s+|m|-3)$ disjoint spheres of square $-2$ avoiding the submanifolds used in our construction. Tubing the 2--handle cores into these spheres has the same effect as blowing up, without changing $X_k$. When the integers $r$, $s$ and $|m|$ are also allowed to be even but not 2, we need an additional trick. For $n\ge3$ we locate an $E(2)$ fiber summand in $X_0$ away from the construction site, then cut it out and reglue it by a cyclic permutation of the circle factors of the boundary 3--torus. This modifies the manifolds $X_k$ so that they each contain three (and more) disjoint spheres of square $-3$, made from sections of $E(2)$ by capping off with vanishing cycles of $E(n-2)$. Using these along with our previous even spheres allows us to realize any positive values of $r$, $s$ and $|m|$ except 2. The manifolds remain pairwise nondiffeomorphic by a useful result of Sunukjian \cite{S}. (This shows that manifolds made by the Fintushel--Stern construction on a given manifold $X_0$ are distinguished by the associated Alexander polynomials, in spite of subtleties introduced by automorphisms of $\Z[H_2(X)]$.) Irreducibility follows by examining the Seiberg--Witten basic classes. These are all linear combinations of the fiber classes of the two elliptic summands (by Doug Park \cite[Corollary~22]{P} for $X_0$, extended to each $X_k$ by the Fintushel--Stern formula \cite{FS}). Thus, all differences of basic classes have square 0. However, if any $X_k$ were reducible, it would split off a negative definite summand carrying a homology class $e$ with square $-1$. Any basic class $c$ would have nonzero (odd) value on $e$. By the gluing formula of \cite[Theorem~14.1.1]{Fr}, reversing the sign of $\langle c,e\rangle$ would give a new basic class $c'$ with $(c-c')^2$ negative.
\end{proof}

\begin{Remarks}\label{even} (a) This irreducibility argument misses the case with $r$, $s$ or $|m|$ equal to 2, for the technical reason that a disjoint sphere of square $-1$ would mean our starting manifold was reducible. It seems reasonable to conjecture that irreducibility is still attainable by a different method in this case.

(b) Each of our $\Z$--corks $(C(r,s;m),f)$ ($r,s>0>m$) generates many other similar families of closed manifolds. We can vary the starting manifold $X_0$ and distinguish the resulting manifolds $X_k$ from each other by Sunukjian's result, then distinguish various families from each other by their Seiberg--Witten invariants. Alternatively, since our construction only uses a single clasp of the knots $K_k$, we can apply the construction to other families of knots (or links) related by the twisting of a clasp described by Figure~\ref{torus}.

(c) Our corks $C(r,s;m)$ all have Mazur type, built with a single handle of each index 0, 1 and 2. This is because they have the form $C(\kappa,m)$ for a 2--bridge knot $\kappa$ (namely $\kappa(r,-s)$). The complement $P$ of $\kappa$ then has a handle decomposition with two 1--handles and a 2--handle, as does $I\times P$. The final 2--handle $h_0$ of $C(\kappa,m)$ cancels a 1--handle. (See Figure~3 of \cite{Cork2}.)

(d) Each of these corks also embeds in the 4--sphere. In fact, the double of any $C(\kappa,m)$ is also obtained from the complement of the spin of $\kappa$ by filling trivially to get $S^4$ (for even $m$) or by Gluck filling (which gives $S^4$ for all spun knots \cite{Gl}). We are left with the following question, which can be restated as a problem about certain torus surgeries in $S^4$: 
\end{Remarks}

\begin{ques} Does iterated twisting on these corks in $S^4$ always give the standard $S^4$?
\end{ques}

\end{document}